\input ./style/arxiv-general.cfg
\documentclass[MSNbibl,number,citesort,seceqn,dvips]{arxbj}
\makeatletter
   \@ifpackageloaded{graphicx}{}{\usepackage{graphicx}}
\makeatother
\usepackage{upgreek}

\aid{0}
\volume{21}
\issue{4}
\pubyear{2015}
\firstpage{2457}
\lastpage{2483}
\doi{10.3150/14-BEJ651} 
\docsubty{FLA}

\makeatletter
\newcommand{\frace}[2]{{#1}/{(#2)}}
\renewcommand{\emptyset}{\varnothing}
\renewcommand{\pi}{\uppi}
\newcommand{\binom}[2]{{{#1}\choose{#2}}}
\newproclaim{defn}{Definition}
\newremark{rem}[defn]{Remark}
\newtheorem{prop}[defn]{Proposition}
\newremark{example}[defn]{Example}
\newtheorem{cor}[defn]{Corollary}
\newtheorem{lem}[defn]{Lemma}
\newtheorem{them}[defn]{Theorem}

\newcommand{\wP}{\mathbb{P}}
\newcommand{\wN}{\mathbb{N}}
\newcommand{\kF}{\mathcal{F}}
\newcommand{\kB}{\mathcal{J}}
\newcommand{\kD}{\mathcal{D}}
\newcommand{\kG}{\mathcal{G}}
\newcommand{\kS}{\mathcal{S}}
\newcommand{\kL}{\mathcal{L}}
\newcommand{\kOS}{\mathcal{OS}}
\newcommand{\kOL}{\mathcal{OL}}
\newcommand{\stern}{\ast}
\newcommand{\N}{\mathbb{N}}
\newcommand{\R}{\mathbb{R}}
\renewcommand{\P}{\mathbb{P}}
\newcommand{\A}{\mathcal{A}}
\newcommand{\E}{\mathbb{E}}
\newcommand{\B}{\mathcal{J}}
\newcommand{\F}{\mathcal{F}}
\newcommand{\K}{\mathcal{K}}
\makeatother

\begin{document}
\begin{frontmatter}

\title{A new class of large claim size distributions: Definition,
properties, and ruin theory}
\runtitle{A new class of large claim size distributions}

\begin{aug}
\author[A]{\inits{S.}\fnms{Sergej}~\snm{Beck}\thanksref{e1}\ead[label=e1,mark]{sergej.beck@hotmail.de}},
\author[A]{\inits{J.}\fnms{Jochen}~\snm{Blath}\corref{}\thanksref{e2}\ead[label=e2,mark]{blath@math.tu-berlin.de}} \and
\author[A]{\inits{M.}\fnms{Michael}~\snm{Scheutzow}\thanksref{e3}\ead[label=e3,mark]{ms@math.tu-berlin.de}}
\address[A]{Institut f\"ur Mathematik, Technische Universit\"at
Berlin, Stra{\ss}e des 17. Juni 136, 10623 Berlin, Germany. \printead{e1,e2,e3}}
\end{aug}

\received{\smonth{7} \syear{2013}}
\revised{\smonth{2} \syear{2014}}

%
\begin{abstract}
We investigate a new natural class $\mathcal{J}$ of probability
distributions modeling large claim sizes, motivated by the `principle
of one big jump'.
Though significantly more general than the (sub-)class of
subexponential distributions $\mathcal{S}$, many important and
desirable structural properties can still be derived.
We establish relations to many other important large claim distribution
classes (such as $\mathcal{D}$, $\mathcal{S}$, $\mathcal{L}$,
$\mathcal
{K}$, $\mathcal{OS}$ and $\mathcal{OL}$),
discuss the stability of $\mathcal{J}$ under tail-equivalence,
convolution, convolution roots, random sums and mixture, and then apply
these results to derive a partial analogue of the famous
Pakes--Veraverbeke--Embrechts theorem from ruin theory for $\mathcal{J}$.
Finally, we discuss the (weak) tail-equivalence of infinitely-divisible
distributions in $\mathcal{J}$ with their L\'evy measure.
\end{abstract}

%
\end{frontmatter}

\section{Introduction}
\label{sn:intro}


Large claim size distributions play an important role in many areas of
probability theory and related fields, in particular insurance and
finance. They often describe `extreme events'
and are typically `heavy-tailed' (see, e.g., \cite{Embrechts1997} for
an overview). However, the class of heavy-tailed random variables $\K$
(defined in Section~\ref{ssn:relation} below) has a very
rich structure, and the identification and discussion of relevant
sub-classes is still an area of active research (see, e.g.,
\cite{Foss2011} for a recent account).
While this makes it difficult to formulate general statements for $\K$,
for example regarding ruin probabilities, such results can be achieved
for certain important subclasses,
most importantly the \emph{subexponential distributions} $\kS$. Recall
that the distribution $F$ of i.i.d. nonnegative random variables $X_1,
X_2, \ldots$ is called
subexponential, iff
%
%
\begin{equation}
\label{eq:subexp} \lim_{x\to\infty}\frac{\P(\max(X_{1},\ldots,X_{n})>x)}{\P
(X_{1}+\cdots
+X_{n}>x)}=1
\end{equation}
for every $n \ge2$. This means that the tail of the distribution of
the maximum of $n$
such random variables is asymptotically equivalent to the tail of the
distribution of
their sum. Hence, this sum is typically dominated by its largest
element in the case of
an extreme event.

The class $\kS$ of subexponential distributions has several important
stability properties, and in particular allows an elegant
characterization of the asymptotic behaviour of the
ruin probability in the Cram\'er--Lundberg model (and in a weaker form
also for more general renewal models). Indeed, the corresponding ruin
function $\Psi$ is asymptotically equivalent,
for large initial capital, to the so-called tail-integrated
distribution $F_I$ associated with $F$ (suitably normalized), iff
$F_I\in\kS$ (e.g., \cite{Veraverbeke82}, see also Theorem~\ref
{thm:P-V-E} below).

From an intuitive point of view, one might ask whether condition (\ref
{eq:subexp}) on the tail behaviour of the $X_1, X_2, \ldots$ might be
too restrictive and could be weakened. For example, one could require
that the maximum is sufficiently close to, but not quite at the same
level as, the sum of the claim sizes $x$, say greater than $x-K$ for
some constant $K$. This appears to be a natural definition of the
folklore `principle of one big jump' and leads to the following definition.

%
%
\begin{defn}[(Distributions of class $\kB$)]
Let $X_1, X_2, \ldots$ be a sequence of i.i.d. nonnegative
(essentially) unbounded random variables. Let $F$ denote their common
distribution function and let $\F$ denote the set of all distribution
functions of nonnegative random variables with unbounded support. We
define the class $\kB\subset\F$ as the set of all distribution
functions $F\in\F$, such that
for all $n \ge2$,
%
%
\begin{equation}
\label{eq:classB} \lim_{K \to\infty} \liminf_{x \to\infty}
\frac{\P(\max(X_1,
\ldots,
X_n)>x-K,X_1 + \cdots+ X_n > x)}{\wP(X_1 + \cdots+ X_n > x)}=1.
\end{equation}
\end{defn}

Just as in the case of subexponential distributions, it is enough that
(\ref{eq:classB}) holds for $n=2$
(see Proposition~\ref{lem:basic-properties}). We will provide several
equivalent
formulations of (\ref{eq:classB}) below.
Of course, such a definition raises immediately a variety of questions.
First, one certainly needs to clarify whether this definition produces
a nontrivial new class of distributions at all. We will answer this
affirmatively in Section~\ref{ssn:relation} where we will
also discuss the relation of $\B$ to other distribution classes (it is
obvious from the definition that $\kS\subset\B$). At first sight
rather surprisingly, it turns out that our definition also admits some
light-tailed distributions to $\B$, see Example~\ref{ex:light} (which,
as an element of $\mathcal{S}(\gamma)$ with $\gamma=1$, is also known
to obey a `principle of one big jump').
Given this last fact, a second natural question is whether $\B$ is
still sufficiently coherent to
exhibit convenient closure properties. It turns out that $\B$ is closed
under weak tail-equivalence (in contrast to $\kS$), see Proposition~\ref{TailClosure-Lemma} below,
and has good properties with respect to closure under convolution and,
importantly, convolution roots (Proposition~\ref{Convolution-Lemma}),
as well as mixture (Proposition~\ref{lem:mix}). The same holds true for
random sums (Propositions~\ref{lem:Zufaellige-Sum-B} and~\ref{lem:Zufaellige-Sum-B-2}).

These rather remarkable properties will then be applied in
Section~\ref{sec:appl}, where we provide a partial analogue of the
Pakes--Veraverbeke--Embrechts theorem for $\B$
(Theorem~\ref{thm:P-V-E-for-B}), establishing weak tail-equivalence
among classical
risk quantites from ruin theory. This result is new and appears quite
striking, given that the class $\mathcal{J}$ is far richer than
$\mathcal{S}$.

Finally, for infinitely-divisible elements of $\B$, we prove their weak
tail-equivalence with
their normalized L\'evy measure, in the spirit of earlier results of
Goldie \textit{et al.} \cite{Embrechts1979} and Shimura and Watanabe
\cite{Shimura2005}.

%
%
\begin{rem}\label{rem:classA}
Regarding the rationale behind (\ref{eq:classB}) one might wonder
whether one should also consider distribution functions $F$ with the
property that
%
%
\begin{equation}
\label{eq:classA} \lim_{x \to\infty} \frac{\P(\max(X_1, \ldots,
X_n)>(1-\varepsilon
)x,X_1+ \cdots+X_{n}>x)}{\P(X_1 + \cdots+ X_n > x)}=1,
\end{equation}
for all $n \ge2$, and
for all $\varepsilon\in(0,1)$. Indeed, this natural condition gives
rise to an even larger class of distributions, denoted by $\A$, with
$\kS\subset\B\subset\A$.
Some results for the class $\A$ can be found in the dissertation of
Beck \cite{Beck2014}.
\end{rem}

%
%

\section{Basic properties of the class $\kB$}
\label{sn:relation}

\subsection{Notation and set-up}
\label{ssn:notation}

Throughout Section~\ref{sn:relation}, we let $X_1, X_2, \ldots$ be a
sequence of i.i.d. random variables on some probability space $(\Omega,
\A, \P)$ with values in $[0,\infty)$. Let $F \in\mathcal{F}$ denote
their common distribution function.
We denote by $S_n$ the sum of the first $n$ random variables, that is
\[
S_{n}= \sum_{i=1}^{n}
X_{i}. %
\]
Further, let $\overline{F}(x):=1-F(x)$ be the
\emph{tail} of $F$. If $\nu$ is a probability measure on $[0,\infty)$,
then we define $\overline\nu(t):=\nu((t,\infty))$.
Let $F\ast G$ be the convolution
of two distribution functions $F, G \in\F$ and $F^{n\ast}$, for $n
\ge0$, the $n$-fold convolution of $F$ with itself, where
$F^{1\ast}:=F$ and $F^{0\ast}$
is the distribution corresponding to the Dirac measure at 0.
Let $f$ and $g$ be two positive functions on $[0,\infty)$.
We write that $f\sim g$ if
\[
\lim_{x\rightarrow\infty} \frac{f(x)}{g(x)}=1, %
\]
that is, $f$ and $g$ are \emph{(strongly) asymptotically equivalent}
(as $x \to\infty$), and $f\asymp g$ in the case
\[
0<\liminf_{x\rightarrow\infty} \frac{f(x)}{g(x)} \le\limsup
_{x\rightarrow\infty} \frac{f(x)}{g(x)} < \infty. %
\]
The latter relation will be called \emph{weak asymptotic equivalence}.
Whenever $\bar\F\subseteq\F$, we freely write $X \in\bar\F$ or
$\mu
\in\bar\F$
for a nonnegative random variable $X$ or a probability measure $\mu$ on
$[0,\infty)$ when the associated distribution function belongs to
$\bar
\F$.
Let $\mathcal{G}$ denote the set of nonnegative, unbounded and
nondecreasing functions. Finally, for all $n \in\N$ and for all $k
\in
\{1, \ldots, n\}$, $x_{k,n}$ denotes
the $k$th largest among $x_1, \ldots, x_n$.

\subsection{Equivalent characterizations of the class $\mathcal{J}$}
\label{ssn:charact}

It is interesting to see that the defining relation (\ref{eq:classB})
is only one of many ways to characterize the class $\kB$. Define, for
$n \ge2$,
%
%
\begin{eqnarray}
\label{defnBwithG} \mathcal{J}_{1}^{(n)} & := & \Bigl\{ F\in
\mathcal{F}:\lim_{K\rightarrow
\infty}\liminf_{x\rightarrow\infty}
\wP(X_{2,n}\leq K|S_{n}>x)=1 \Bigr\} ,
\nonumber
\\
\mathcal{J}_{2}^{(n)} & := & \Bigl\{ F\in\mathcal{F}:\lim
_{K\rightarrow
\infty}\inf_{x\geq0} \wP(X_{2,n}\leq
K|S_{n}>x)=1 \Bigr\} ,
\nonumber
\\
\mathcal{J}_{3}^{(n)} & := & \Bigl\{ F\in\mathcal{F}:\lim
_{K\rightarrow
\infty}\liminf_{x\rightarrow\infty}\wP (X_{1,n}>x-K|S_{n}>x)=1
\Bigr\} ,
\\
\mathcal{J}_{4}^{(n)} & := & \Bigl\{ F\in\mathcal{F}:\lim
_{K\rightarrow
\infty}\liminf_{x\rightarrow\infty}\wP (X_{1,n}>S_{n}-K|S_{n}>x)=1
\Bigr\},
\nonumber
\\
\mathcal{J}^{(n)} & := & \Bigl\{ F\in\mathcal{F}:\lim
_{x\rightarrow
\infty}\wP\bigl(X_{2,n}>g(x) |S_{n}>x\bigr)=0\
\forall g\in\mathcal {G} \Bigr\}.
\nonumber
\end{eqnarray}
Note that, by definition, $\mathcal{J}= \bigcap_{n\ge2} \mathcal
{J}_3^{(n)}$.
However, it turns out that all of the above subclasses are equal to
the class $\B$. Indeed, we have

%
\begin{prop}
\label{lem:basic-properties}
For all $n \ge2$,
\begin{enumerate}[(b)]
\item[(a)]$\kB^{(n)}=\kB_{1}^{(n)}=\kB_{2}^{(n)}=\kB_{3}^{(n)}=\kB
_{4}^{(n)}$,

\item[(b)]$\kB^{(n)}=\kB^{(n+1)}$,

\item[(c)]$\B=\B^{(2)}$.
\end{enumerate}
\end{prop}

A proof can be found in Section~\ref{sn:proofs}. Note that a term
reminiscent to the one in the definition of class
$\B^{(2)}$ appears implicitly in \cite{Asmussen2003}, Proposition~2.

%
\begin{rem}
An elegant probabilistic way to think about the condition giving rise
to class $\B^{(n)}_2$ is to interpret it as tightness condition of the
conditional laws of $X_{2,n}$, given $S_n>x$.
\end{rem}

\subsection{Relation to other classes of claim size distributions and
heavy tails}
\label{ssn:relation}

Recall that a claim size distribution $F\in\F$ is called \emph{heavy-tailed},
if it has no exponential moments, that is,
\[
\int_{0}^{\infty}\mathrm{e}^{\lambda x}\,
\mathrm{d}F(x)=\infty \qquad\mbox{for all }\lambda>0.
\]
In this case, we write $F\in\mathcal{K}$. Following
the definition (but not the notation) of \cite{Yang2005}, we write
$F\in
\mathcal{K}^*$ if $\lim_{x\rightarrow\infty}\mathrm{e}^{\lambda
x}\overline{F}(x)=\infty$
holds for all $\lambda>0$. Note that $\mathcal{K}^*\subsetneq
\mathcal{K}$,
see, for example, \cite{Su2009}, and thus we call elements of
$\mathcal{K}^*$
`strongly heavy tailed'.

Three of the most important and well-studied subclasses of heavy-tailed
distributions are the class $\kS$ of \emph{subexponential}
distributions, the class of \emph{long-tailed} distributions $\kL$ and
the class $\kD$ of \emph{dominatedly varying} distributions. Recall
that a distribution
$F \in\F$ is subexponential if for all $n \ge2$,\vspace*{1pt}
\[
\lim_{x\rightarrow\infty}\frac{\overline{F^{n\stern
}}(x)}{\overline{F}(x)}=n %
\]
(it is actually enough to require this condition for $n=2$ only, see,
e.g., \cite{Embrechts1997}), and that $F$ is long-tailed if\vspace*{1pt}
\[
\lim_{x\rightarrow\infty} \frac{\overline{F}(x+y)}{\overline{F}(x)}=1 %
\]
for every $y\in\mathbb{R}{\setminus}\{0\}$ (or equivalently for
some). Further, $F$ has a dominatedly varying tail, if
\[
\limsup_{x\rightarrow\infty} \frac{\overline{F}(xu)}{\overline
{F}(x)}<\infty %
\]
for all (or equivalently for some) $0<u<1$.
It is well known that
\[
\kS\subset\kL\subset\mathcal{K}^*\subset\mathcal{K},\qquad\kD \subset
\mathcal{K}^*\subset \mathcal{K},\qquad\kL\cap\kD\subset\kS\quad\mbox{and}\quad\kD
\nsubseteq\kS,\kS \nsubseteq\kD,
\]
see
\cite{Embrechts1997}
for most of these inclusions (the remaining ones are easy to check).

A generalization of the class of subexponential distributions is given
by Shimura and Watanabe \cite{Shimura2005}. They systematically
investigate the class $\kOS$ of `$O$-\emph{subexponential}'
distributions, which was introduced by Kl\"uppelberg in \cite
{Klueppelberg1990}, where $F \in\kOS$ if
%
%
\begin{equation}
\label{eq:cf} c_{F}:=\limsup_{x\rightarrow\infty}
\frac{\overline{F^{2\stern
}}(x)}{\overline{F}(x)}<\infty.
\end{equation}
In a similar way, it is possible to generalize the class $\kL$. Let
$\kOL$ be the class of all distributions such that
\[
\limsup_{x\rightarrow\infty} \frac{\overline{F}(x+y)}{\overline
{F}(x)}<\infty %
\]
for every $y\in\R$. The generalizations $\kOL$ and $\kOS$ of the
classes $\mathcal{L}$
and $\mathcal{S}$ contain some light-tailed distributions, so that
$\kOL
, \kOS\nsubseteq\mathcal K$. Further, it can be shown \cite
{Shimura2005}, Proposition~2.1, that
\[
\kOS\subset\kOL. %
\]
Finally, we recall the light-tailed distribution classes $\mathcal{S}
(\gamma)$ and
$\mathcal{L} (\gamma)$, for $\gamma\ge0$: We say that a distribution
$F \in\F$ belongs to $\mathcal{S} (\gamma)$ for some $\gamma\ge0$,
if for any $y \in\R$,
%
%
\begin{equation}
\label{eq:defSgamma} \lim_{x \to\infty} \frac{\overline{F}(x+y)}{\overline{F}(x)} = \exp (-\gamma
y),
\end{equation}
and for some constant $c \in(0, \infty)$,
\[
\frac{\overline{F^{2*}}(x)}{\overline{F}(x)} = 2c < \infty. %
\]
A distribution $F \in\F$ belongs to $\mathcal{L} (\gamma)$, iff it
satisfies (\ref{eq:defSgamma}).
These classes were introduced independently by
Chistyakov \cite{Chistyakov1964} and Chover, Ney
and Wainger \cite{Chover1973a,Chover1973b}, see also \cite
{Embrechts1982}. Note that $\mathcal{L} (0)=\mathcal{L}$ and
$\mathcal
{S} (0)=\mathcal{S}$.

For our new class $\B$, we have the following results.


%
\begin{prop}
\label{lem:Class_B}
\begin{enumerate}[(b$'$)]
\item[(a)]$\mathcal{J}\subset\mathcal{OS}$,

\item[(b)]$\kB\cap\kL=\kS$,

\item[(b$'$)]$\mathcal{J} \cap\mathcal{L} (\gamma) = \mathcal{S}
(\gamma),
\gamma>0$,

\item[(c)]$\mathcal{D}\subset\kB$,

\item[(d)]$\B\nsubseteq\K$.
\end{enumerate}
\end{prop}

A proof can be found in Section~\ref{sn:proofs}.
Part (b$'$) has been suggested to us by Sergey Foss. Note
that (b$'$) already implies (d) (so that the latter is in principle redundant),
but we think that the fact that $\B$ includes some light-tailed functions
is important and thus we end this subsection with a concrete example
(still obeying a `principle of one big jump').

%
%
\begin{example}\label{ex:light} Consider the distribution function
$F\in\kF$ with density
\[
f(x)=\mathrm{e}^{-x}\frac{C}{1+x^{2}},\qquad x\ge0,
\]
for $C>0$ such that $\int_{0}^{\infty}f(x)\,\mathrm{d}x=1$. Note
that there
seems to be no closed-form expression for $C$, but it can be evaluated
numerically to $C\approx1.609$. Obviously $F\notin\K$ and thus
$F\notin
\mathcal{S}$.
Since
\[
\lim_{K\to\infty}\lim_{x\to\infty}\frac{\int_{K}^{x-K}f(y)f(x-y)
\,\mathrm{d}y}{\int_{0}^{x}f(y)f(x-y) \,\mathrm{d}y} =0,
\]
it follows that
\[
\lim_{K\to\infty}\limsup_{z\to\infty}\frac{\int_{z}^{\infty
}\int_{K}^{x-K}f(y)f(x-y) \,\mathrm{d}y \,\mathrm{d}x}{\int_{z}^{\infty
}\int_{0}^{x}f(y)f(x-y) \,\mathrm{d}y
\,\mathrm{d}x}=
\lim_{K\to\infty}\limsup_{z\to\infty}\frac{\wP
(X_{2,2}>K,S_{2}>z)}{\wP(S_{2}>z)}=0
\]
and hence $F\in\kB$. From Proposition~\ref{lem:Class_B}(a), we
also have that $F\in\mathcal{OS}$. Indeed, we can compute
$c_{F}$ from (\ref{eq:cf}) and obtain
%
%
\begin{equation}
\label{eq:C_F_pi} c_{F}=C\pi.
\end{equation}
Note that
$f(x)$ is obtained from the (subexponential)
density $2/(\pi(1+x^{2}))$ by multiplication with a negative exponential
and a suitable constant. This is a typical way to construct distributions
of the distribution class $\mathcal{S}(\gamma)$, $\gamma\ge0$,
and indeed we have $F\in\mathcal{S}(\gamma)$ with index $\gamma=1$.
This class consists of light-tailed functions and has a well-studied
ruin theory, obeys the `principle of one big jump', and is outside the classical
Lundberg framework.

\end{example}

%
%
\begin{rem}
It seems natural to ask `how many' or `which kind of' light-tailed
functions can be found in $\B$. As a first result in this direction
note that
since $\B\subset\mathcal{OL}$, it follows from Proposition~2.2 in
\cite{Shimura2005} that each light-tailed distribution $F \in\B$
exhibits at least some infinite exponential moments, that is, there
exists a $\lambda_F>0$ such that
\[
\int_0^\infty\mathrm{e}^{\lambda_F x} \,
\mathrm{d}F(x) = \infty. %
\]
Hence, the class $\B$ in some sense `touches the boundary' of the
class of light-tailed functions.
In view of (b$'$), the conjecture $\B= (\B\cap\K) \cup(\bigcup_{\gamma>
0} S(\gamma))$ seems attractive.
\end{rem}

\subsection{Closure properties}
\label{ssn:closure}

As a first result, we show that our new class $\B$ is closed under weak
asymptotic tail-equivalence (in contrast to $\kS$ and $\kL$, which
require (strong) asymptotic tail-equivalence for closure).

%
%
\begin{prop}\label{TailClosure-Lemma}
If $F\in\mathcal{J}$ and $\overline{F}\asymp\overline{G}$, then
$G\in\mathcal{J}$.
\end{prop}


%
%
\begin{example}
Neither $\mathcal{L}$ nor $\mathcal{S}$ are closed under weak
tail-equivalence.
Indeed, let
\[
F(x):= \biggl(1-\frac{1}{x} \biggr)^{+},\qquad x \ge0,
\]
be a Pareto distribution with index $1$, so that $F$ is subexponential
and long-tailed. Let $G$ be the `Peter-and-Paul' distribution, that is,
$\overline{G}(x)=2^{-k}$
for $x\in[2^{k},2^{k+1})$, $k \in\N_0$. Then $F$ and $G$ are weakly
tail-equivalent,
that is, $\overline{F}\asymp\overline{G}$, but $G\notin\mathcal{L}$
and hence $G\notin\mathcal{S}$.
\end{example}

Although we will see that $\B$ is not closed under convolution,
we will find below that we have closure for `convolution powers' and
for weakly tail-equivalent distributions. Further, we have closure for
`convolution roots', in contrast to $\kOS$ (cf. \cite{Shimura2005}) --
this property is highly desirable as we will see in the sequel.

We say that a distribution class $\mathcal{C}$ is \emph{closed under
convolution},
if $F_{1}*F_{2}\in\mathcal{C}$ for any $F_{1}, F_{2} \in\mathcal{C}$.
It is well known that the class $\kL$ is
closed under convolution, see \cite{Embrechts1980}, Theorem~3(b).

%
%
\begin{prop}\label{Convolution-Lemma}
\begin{enumerate}[(b)]
\item[(a)]If $F\in\mathcal{J}$, then $\overline{F}\asymp\overline
{F^{n\stern}}$
and hence $F^{n\stern}\in\kB$.

\item[(b)]If $F\in\mathcal{J}$ and $\overline{F}\asymp\overline
{G}$, then
$F\stern G\in\kB$.

\item[(c)]If $F^{n\stern}\in\kB$, then $\overline{F}\asymp
\overline{F^{n\stern}}$
and hence $F\in\kB$.
\end{enumerate}
\end{prop}

%
%
\begin{example}
The classes $\kS$ and $\mathcal{J}$ are not closed under convolution.
A counterexample for the class $\kS$ is given in \cite{Leslie1989}, Section~3.
Since $\kS\subset\kL$, by the counterexample from
\cite{Leslie1989} and Theorem~3(b) from \cite{Embrechts1980}
(convolution closure of $\kL$) we know there exist two distributions
$F_{1}$, $F_{2}$ such that $F_{1},F_{2}\in\kS$ and $F_{1}\ast
F_{2}\in\kL$
but $F_{1}\ast F_{2}\notin\kS$. Since we have
$\kB\cap\kL=\kS$ from Proposition~\ref{lem:Class_B}(b),
$F_{1},F_{2}\in
\kB$ but $F_{1}\ast F_{2}\notin\kB$.
\end{example}

We now turn to mixture properties of the class $\B$.
Let $X,Y$ be two random variables with distribution functions $F,G \in
\F$.
Recall that $X \vee Y$ (resp. $X \wedge Y$) denotes the pointwise
maximum (resp. minimum) of $X$ and $Y$. We call a random variable $Z$
\emph{mixture} of $X$ and $Y$ with parameter $p \in(0,1)$, if its
distribution function is given by
\[
pF+(1-p)G \in\mathcal F.
\]
%
It is easy to see that if $X$ and $Y$ are independent,
we have for all mixtures $Z$ with $p\in(0,1)$,
%
%
\begin{equation}
\label{eq:mixequiv} (X\vee Y)\asymp Z.
\end{equation}
%

%
\begin{prop} \label{lem:mix} Let $
X,Y
\in\mathcal{F}$
be independent.
\begin{enumerate}[(b)]
\item[(a)]If $
X,Y
\in\kB$, then the following are equivalent:
\begin{enumerate}[(iii)]
\item[(i)]$(X\vee Y)\in\kB$;

\item[(ii)]$(X+Y)\in\kB$;

\item[(iii)]$Z\in\kB$.
\end{enumerate}

\item[(b)]If $
X,Y
\in\kB$, then $(X\wedge Y)\in\kB$.
\end{enumerate}
\end{prop}
%

The previous statement remains true when $\kB$ is replaced by
$\kS$ (see \cite{Yakymiv1997}, Theorem~1, \cite{Geluk2009}, Theorem~1,
and \cite{Foss2011}, Theorem~3.33).

%
%
\begin{rem}
Concerning part (b) of the previous proposition
one may ask if $X,Y \in\kB$ even implies that $X\vee Y$ and $X+Y$ are
weakly tail equivalent which would immediately imply
the equivalence of (i) and (ii).
Perhaps surprisingly, this is
not true in general -- not even under the stronger assumption that $X,Y
\in\kS$ as the example in \cite{Leslie1989}
shows.
\end{rem}

\subsection{Random sums}
\label{ssn:rand_sum}

As before, let $F\in\F$ be the common distribution function of the
i.i.d. random variables $\{X_{i}\}$. Recall the notation
\[
c_{F}=\limsup_{x\rightarrow\infty}\frac{\overline{F^{2\ast
}}(x)}{\overline{F}(x)}.
\]
Denote by $N$ a discrete random variable with values in $\N_{0}$,
independent of the $\{X_{i}\}$, with probability weights $p_{n}:=\P\{
N=n\}$, $n\geq0$
and $p_{0}<1$. Denote by $N^{(1)}$ and $N^{(2)}$ two independent
copies of $N$, and write
\[
(p*p)_{n}:=\P\bigl(N^{(1)}+N^{(2)}=n\bigr),\qquad n
\geq0.
\]
We now consider the random sum
\[
S_{N}:=\sum_{i=1}^{N}X_{i}
\]
%
with
distribution function $F_{N}$. Under a suitable decay condition
on the $(p_{n})$, we obtain the following stability property of $\B$
for a random number of convolutions and convolution roots:

%
%
\begin{prop}\label{lem:Zufaellige-Sum-B}
%
\begin{enumerate}[(b)]
\item[(a)]If $F\in\mathcal{J}$ and $\sum^{\infty
}_{k=1}p_{k}(c_{F}+\varepsilon-1)^{k}<\infty$
for some $\varepsilon>0$, then $\overline{F_{N}}\asymp\overline{F}$
and hence $F_{N}\in\mathcal{J}$.

\item[(b)]If $F_{N}\in\mathcal{J}$ and $\sum^{\infty
}_{k=1}p_{k}(c_{F_{N}}+\varepsilon-1)^{k}<\infty$
for some $\varepsilon>0$, then
$\overline{F}\asymp\overline{F_{N}}$
and hence $F\in\mathcal{J}$.
\end{enumerate}
\end{prop}

%
%
\begin{rem}
Note that one cannot infer $F_{N}\in\kB$ or $\overline
{F}\asymp\overline{F_{N}}$ from $F\in\kB$ without additional conditions
on $N$. This is true even if $N$ is a geometric random variable, say
with parameter $p\in(0,1)$ and probability weights
$p_k=p^k (1-p)$, $k \ge0$.
Indeed, while it is obvious that the condition of the Proposition is
satisfied for
all $p \in (0, (c_F+ \varepsilon-1)^{-1} )$,
a counterexample is given by the distribution $F\in\kB$
from Example~\ref{ex:light} with geometric $N$ that has a parameter
$p$ close enough to $1$.
To see this,
consider a sequence of i.i.d. random variables $X_{1},X_{2},X_{3},
\ldots$
with distribution function $F\in\kB$ from Example~\ref{ex:light}.
Let $\alpha>0$ be such that $\alpha\mathbb{E}X_{1}>1$. For all $m\in
\mathbb{N}$
we
have
%
\[
\wP(X_{1}+ \cdots+X_{N}>m) \geq\wP\bigl(N\geq \lfloor
\alpha m \rfloor\bigr)\wP(X_{1}+ \cdots+X_{ \lfloor\alpha m
\rfloor}>m)
\]
and
\[
\wP(X_{1}>m)=\int^{\infty}_{m}
\frac{C\mathrm
{e}^{-x}}{1+x^{2}} \,\mathrm{d}x\leq C\mathrm{e}^{-m}.
\]
If $p$ is close enough to $1$ we obtain
by our choice of $\alpha$ and the law of large numbers,
%
\[
\frac{\wP(X_{1}+ \cdots+X_{N}>m)}{\wP(X_{1}>m)}\geq\frac{
\mathrm{e}^{m}
}{C}p^{ \lfloor\alpha m \rfloor} \wP(X_{1}+
\cdots +X_{
\lfloor\alpha m \rfloor}>m)\rightarrow\infty,
\]
for $m\rightarrow\infty$, so $\overline{F}\asymp\overline{F_{N}}$ does
not hold. It follows from part (b) of
Proposition~\ref{lem:Zufaellige-Sum-B-2}
below
that $F_{N}\notin\kB$.
\end{rem}
%

A related result for random sums
in $\kB$ and $\kOS$ can be obtained under the following condition on
$N$ and $c_F$.

%
%
\begin{prop} \label{lem:Zufaellige-Sum-B-2}
Suppose
\[
\liminf_{n\rightarrow\infty}\frac{\wP(N_{1}+N_{2}>n)}{\wP
(N_{1}>n)}>c_{F_{N}}=\limsup
_{x\rightarrow\infty}\frac{\overline
{F_{N}^{2\stern}(x)}}{\overline{F_{N}(x)}},
\]
then the following assertions hold.
\begin{enumerate}[(b)]
\item[(a)]If $F_{N}\in\kOS$, then there exists $m\in\mathbb{N}$ such
that
$\overline{F^{m\stern}}\asymp\overline{F_{N}}$.

\item[(b)]If $F_{N}\in\kB$, then
$\overline{F}\asymp\overline{F_{N}}$
and
hence $F\in\kB$.
\end{enumerate}
\end{prop}

In \cite{Watanabe} it is pointed out that $\lim_{n\to\infty}\frac
{p_{n+1}}{p_{n}}=0$
implies $\lim_{n\to\infty}\frac{(p*p)_{n}}{p_{n}}=\infty$,
which
in turn implies $\liminf_{n\rightarrow\infty}\frac{\wP
(N_{1}+N_{2}>n)}{\wP(N_{1}>n)}=\infty$.
From there, we also recall some examples for $N$.


%
%
\begin{example}
The following distributions satisfy the condition
$\liminf_{n\rightarrow\infty}\frac{(p*p)_{n}}{p_{n}}=\infty$:
\begin{enumerate}[(3)]
\item[(1)]Poisson distribution: $p_{n}=\frac{c^{n}}{n!}\mathrm{e}^{-c}$,
$c>0$.

\item[(2)]Geometric distribution: $p_{n}=(1-p)p^{n}$, $p \in(0,1)$.

\item[(3)]Negative Binomial distribution: $p_{n}=
\binom{n+r-1}{n}
p^{n}(1-p)^{r}$,
$p\in(0,1)$, $r>0$.
\end{enumerate}
\end{example}

\section{Applications}
\label{sec:appl}

\subsection{Ruin theory and maximum of a random walk}
\label{ssn:ruin}


Let $X_1, X_2,\ldots$ be a family of strictly positive i.i.d. random
variables on a probability space $(\Omega, \A, \P)$ with distribution
function $F_X$ and finite
expectation $\mu_X$. Let $N=\{N(t), t \ge0\}$ be a renewal process
with i.i.d. strictly positive waiting times $W_1, W_2, \ldots$ We
assume that the $W_i$ are independent of the $X_i$, and with finite
expectation $1/\lambda$, for some $\lambda> 0$.
We then define the \emph{total claim amount process} as
\[
S(t):= \sum_{i=1}^{N(t)} X_i,
\qquad t\ge0.
\]
Let $T_n:=W_1 + \cdots+ W_n, n\ge1$ be the arrival times of the
claims, where we set $T_0:=0$.
By $c>0$ we denote the \emph{premium rate} and by $u\ge0$ the \emph
{initial capital}. Finally, we define the risk process, for $u\geq0$, by
\[
Z(t):= u+ct -S(t),\qquad t\ge0.
\]
If the above \emph{claim arrival process} $N$ is a Poisson process, we
are in the classical \emph{Cram\'er--Lundberg} model, otherwise, we are
in the more general \emph{Sparre Andersen} model.
By
\[
\tau:= \inf\bigl\{ k \ge1: Z(T_k) < 0\bigr\}
\]
we denote the \emph{ruin time} (with the usual convention that $\inf
\emptyset= \infty$), and by
\[
\Psi(u) := \P\bigl\{ \tau< \infty| Z(0)=u\bigr\}
\]
we denote the \emph{ruin probability}. This quantity is the central
object of study in \emph{ruin theory}. We will be interested in
obtaining asymptotic results for $\Psi(u)$ for large $u$ if $F_X \in
\B
$. To this end, we first reformulate the classical ruin problem into a
question about the maximum of an associated random walk with negative
drift. We follow the exposition of \cite{Yang2011}. Let
\[
\bar X_k := X_k - c W_k =-
\bigl(Z(T_k)-Z(T_{k-1}) \bigr),\qquad k \ge 1.
\]
Note that the $\bar X_i$ are i.i.d. with values in $\R$. We denote
their distribution function by $F_{\bar X}$. By the strong law of large
numbers, we have
$\Psi(u) \equiv1$ if $\E[\bar X_k] \ge0$ (unless $\bar X_1 \equiv0$).
Otherwise, we say that the \emph{net profit condition} holds, and we
denote $a := \E[\bar X_k] < 0$. Let
\[
\bar S_n:= \bar X_1 + \cdots+ \bar X_n,
\qquad n \ge1, \bar S_0:=0.
\]
Under the net profit condition, this is a discrete-time random walk
with negative drift. For the ruin probability, we obtain
\[
\Psi(u) = \P \Bigl\{ \sup_{n\ge0} \bar S_n > u
\Bigr\},\qquad u \ge0.
\]
%
%
Hence we have expressed the probability of ruin in terms of the
distribution of the supremum of a random walk with negative drift,
which is the object that we will now investigate. Let
\[
M:=\sup_{n\geq0} \bar S_{n}
\]
be the supremum of the random walk and denote its distribution by
$F_M$. With this notation, we have $\Psi(u) = \overline{F_M}(u)$,
$u\ge0$.

Denote by
\[
\tau_{+}:=\inf\{n\geq1:\bar S_{n}>0\},
\]
the first passage time over zero, with the convention $\inf
\emptyset=\infty$.
Then, the first ascending ladder
height is given by $\bar S_{\tau_{+}}$, which is a defective random
variable (since $\tau_{+}$ may be infinite).
Set
\[
\P(\tau_{+} < \infty)=:p <1. %
\]
%
We assume that $p>0$ which only excludes uninteresting cases and is
automatically satisfied in the Cram\'er--Lundberg model. Further, let
\[
\overline{G}(x):=\wP(\bar S_{\tau_{+}}>x | \tau_{+}<\infty),
\qquad x \ge0.
\]
It is well known (see \cite{Feller71}, Chapter XII) that the tail of
the distribution $M$ can be calculated by the formula
%
%
\begin{equation}
\label{eq:W_G_relation} \overline{F_M}(x)=(1-p) \sum
_{n=0}^{\infty} 
p^{n}
\overline{G^{n*}}(x),\qquad x \ge0.
\end{equation}
Finally, denote by
\[
F_{I}(x):=1- \min \biggl\{ 1,\int^{\infty}_{x}
\overline {F_{\bar{X}}}(y)\,\mathrm{d}y \biggr\},\qquad 
x
\ge0,
\]
the \emph{tail-integrated} distribution of $F_{\bar X}$.
Then, the classical Pakes--Veraverbeke--Embrechts theorem
can be stated as follows (see \cite{Yang2011}).

%

%
%
\begin{them}[(Pakes--Veraverbeke--Embrechts)] 
\label{thm:P-V-E} 
With the above notation and assumptions,
recalling $a:=\E[\bar{X}_{k}]<0$,
the following assertions are
equivalent:
\begin{enumerate}[(3)]
\item[(1)]$F_{I}\in\kS$;

\item[(2)]$G\in\kS$;

\item[(3)]$F_{M}\in\kS$;

\item[(4)]$\overline{F_{M}}\sim-\frac{1}{a}\overline{F_{I}}$.
\end{enumerate}
\end{them}

Our main goal in this section is to (partially) extend this result from
the class $\mathcal{S}$ to $\kB$. Recall the notation
\[
c_{G}=\limsup_{x\rightarrow\infty}\frac{\overline
{G^{2*}}(x)}{\overline{G}(x)},
\]
%
and from \cite{Yang2011}, Lemma~2.2, that if $F_I \in\mathcal{OL}$,
then $\overline G \asymp\overline{F_I}$.

%
%
\begin{them} 
\label{thm:P-V-E-for-B} With the above notation and $a<0$,
assume additionally that $F_{I}\in\kOL$ and that one of the following
conditions holds:
\begin{enumerate}[(iii)]
\item[(i)]$p(c_{G}+\varepsilon-1)<1$ for some $\varepsilon>0$,

\item[(ii)]$F_{M}\in\kOS$,

\item[(iii)]$F_{I}\in\kB\cap\mathcal{K}^*$.
\end{enumerate}
Then the following assertions are equivalent:
\begin{enumerate}[(3)]
\item[(1)]$F_{I}\in\kB$;

\item[(2)]$G\in\kB$;

\item[(3)]$F_{M}\in\kB$.
\end{enumerate}
Each one of (1), (2) or (3)
combined with each one of
(\textup{i}), (\textup{ii}) or (\textup{iii}) implies
\begin{enumerate}[(4)]
\item[(4)]$\overline{F_{M}}\asymp\overline{G}\asymp\overline{F_{I}}$.
\end{enumerate}
\end{them}

For a (nontrivial) example of a distribution $F_I \in\mathcal{OL}
\cap\B\cap\mathcal{K}^*$, but $\F\notin\mathcal{L}$, see the
recent article \cite{Xu2014}.

\begin{cor}
If $a<0$ and $\overline{F_{I}}\asymp\overline{H}$ for some $H\in
\mathcal{S}$,
then $\overline{F_{M}}\asymp\overline{G}\asymp\overline{F_{I}}$.
\end{cor}

Note that the analogous weak tail-equivalence does not hold if $F_{I}$
is an exponential
distribution with parameter $\lambda>0$, while one has strong asymptotic
tail-equivalence (up to a constant) in the case $F_{I}\in\mathcal
{S}(\gamma)$,
see \cite{Korshunov1997}.

Theorem~\ref{thm:P-V-E-for-B} is inspired by and should be compared
with the recent partial generalization of Theorem~\ref{thm:P-V-E}
to the even larger class $\kOS$ by Yang and Wang in \cite{Yang2011},
Theorems~1.2 and~1.3,
which is however considerably weaker.
In particular, it does not cover Example~3.2 in \cite{Xu2014}.
One reason is that the class $\kOS$ is not closed under convolution
roots, as opposed to $\kS$
and $\kB$.

Let $H^{+}$ denote the positive part of a distribution
function $H$.

%
%
\begin{them} 
\label{thm:P-V-E-for-OS} With the above notation
and $a<0$,
if
$F_{I}\in\kOL$, then
\begin{enumerate}[(b)]
\item[(a)]$\limsup_{x\rightarrow\infty}\frac{\overline
{F_{I}}(x)}{\overline
{F_{M}}(x)}<\infty$;

\item[(b)]
\begin{enumerate}[(ii)]
\item[(i)]$F_{I}\in\kOS$ and

\item[(ii)]$G\in\kOS$ are equivalent;
\end{enumerate}

\item[(c)]
\begin{enumerate}[(iii)]
\item[(iii)]
$\overline{F_{M}}\asymp\overline{G}\asymp\overline{F_{I}}$
yields (\textup{i}) or (\textup{ii}).
\end{enumerate}
\end{enumerate}

If $F_{I}\in\kL$ and $(c_{F_{I}^{+}}-1)<a$, then (\textup{i}) or
(\textup{ii})
yields (\textup{iii}). In this case, (\textup{i}) ((\textup{ii}) or
(\textup{iii})) implies $F_{M}\in\kOS$.
\end{them}

Note that the weak asymptotic tail equivalence (iii)
requires $F_{I}\in\kL$ as opposed to the situation in
Theorem~\ref{thm:P-V-E-for-B}.
Further, (a) gives only a lower asymptotic bound for $\overline
{F_{M}}$ in terms of $\overline{F_{I}}$.


\subsection{Infinitely divisible laws}
\label{sec:invdiv}

In this section, we consider the relation between the asymptotic tail
behaviour of infinitely divisible laws and their L\'evy measures.
Following \cite{Shimura2005}, we denote by $\mathcal{ID}_+$ the class
of all infinitely divisible
distributions $\mu$ on $[0,\infty)$ with Laplace transform
\[
\hat{\mu}(s)=\exp \biggl\{\int^{\infty}_{0}\bigl(
\mathrm {e}^{-st}-1\bigr)\nu(\mathrm{d}t) \biggr\},
\]
where the L\'evy measure $\nu$ satisfies $\overline{\nu}(t)>0$ for
every $t>0$, and
\[
\int^{\infty}_{0}(1\wedge t)\nu(\mathrm{d}t)<\infty.
\]
Define the normalized L\'evy measure $\nu_{1}$ as $\nu_{1}=1_{\{x>1\}
}\nu/\nu(1,\infty)$.
Embrechts \textit{et al.} proved in \cite{Embrechts1979}, Theorem~1, the
following classical result.

%
%
\begin{them}
\label{thm:inf_div_s}
Let $\mu$ be a distribution in $\mathcal{ID}_+$ with L\'evy measure
$\nu$. Then the following assertions are equivalent:
\begin{enumerate}[(3)]
\item[(1)]$\mu\in\kS$;

\item[(2)]$\nu_{1}\in\kS$;

\item[(3)]$\overline\mu\sim\overline\nu$.
\end{enumerate}
\end{them}

Shimura and Watanabe partially extended the result of Embrechts from the
class $\kS$ to the class $\kOS$ in \cite{Shimura2005}, Theorem~1.1:

%
%
\begin{them}
\label{thm:inf_div_os}
Let $\mu$ be a distribution in $\mathcal{ID}_+$ with L\'evy measure
$\nu$.
\begin{enumerate}[(b)]
\item[(a)]The following are equivalent:
\begin{enumerate}[(2)]
\item[(1)]$\nu_1\in\kOS$;

\item[(2)]$\overline\mu\asymp\overline{\nu_1}$.
\end{enumerate}

\item[(b)]The following are equivalent:
\begin{enumerate}[(3)]
\item[(1)]$\mu\in\kOS$;

\item[(2)]$\nu_{1}^{n*}\in\kOS$ for some $n\geq1$;

\item[(3)]$\overline\mu\asymp\overline{\nu_{1}^{n*}}$ for some
$n\geq1$.
\end{enumerate}

\item[(c)]If $\nu_{1}$ is in $\kOS$, then $\mu$ is in $\kOS$. The converse
does not hold.
\end{enumerate}
\end{them}

Since the class $\B$ is closed under convolution roots, one expects to
be able to improve the
result for $\kOS$ to class $\B$ significantly. Indeed this is possible.

%
%
\begin{them}
\label{thm:Infinite-divisibility-and-B}
Let $\mu$ be a distribution in $\mathcal{ID}_+$ with L\'evy measure
$\nu$.
\begin{enumerate}[(b)]
\item[(a)]Then the following assertions are equivalent:
\begin{enumerate}[(2)]
\item[(1)]$\mu\in\kB$;

\item[(2)]$\nu_{1}\in\kB$.
\end{enumerate}

\item[(b)]If (1) or (2) holds, then ${\overline\mu}\asymp{\overline
\nu_{1}}$.
\end{enumerate}
\end{them}


Since the proof is simple, we refrain from postponing it to the next
section and state it here.

\begin{pf*}{Proof of Theorem~\ref{thm:Infinite-divisibility-and-B}}
(a) From Theorem~\ref{thm:inf_div_os}(b),
$\kB\subseteq\kOS$, Proposition~\ref{TailClosure-Lemma}, $\mu\in
\kB$
we infer $\overline{\mu}\asymp\overline{\nu_{1}^{n*}}$ and $\nu
_{1}^{n*}\in\kB$
for some $n\geq1$. The equivalence $\mu\in\kB \Leftrightarrow \nu_{1}\in\kB$ follows immediately from Proposition~\ref
{Convolution-Lemma}(c).

(b) If (1) holds the assertion follows from Theorem~\ref{thm:inf_div_os}(b),
Propositions~\ref{TailClosure-Lemma} and~\ref{Convolution-Lemma}(c).
If (2)
holds, then the assertion follows from Theorem~\ref{thm:inf_div_os}(a)
and $\kB\subset\kOS$.
\end{pf*}

\section{Proofs}
\label{sn:proofs}

Throughout the proofs, we will use the following notation. Denote by
$X,X_{1},X_{2},\ldots$ i.i.d. random variables with
common distribution function $F\in\F$, and by ${Y,Y_{1},Y_{2},\ldots}$
i.i.d. random variables with common distribution function
$G\in\F$. By $X_{k,n}$ we denote the $k$th largest
element (pointwise) out of $X_{1},\ldots,X_{n}$, $1\le k\le n$, and
by $X_{k,(l,\ldots,m)}$ the $k$th largest element (pointwise) out
of $X_{l},\ldots,X_{m}$, $1\le l\le m,1\le k\le m-l+1$. Further,
let
\[
S_{n}:=\sum^{n}_{k=1}X_{{k}}
\quad\mbox{and}\quad\hat {S}_{n}:=\sum^{n}_{k=1}Y_{k}.
\]
Finally, denote by $S_{n}^{(i)}$,
respectively $\hat{S}_{n}^{(i)}$, $i=1,\ldots,4$, independent
identically distributed copies of $S_{n}$ respectively $\hat{S}_{n}$.
We begin with several technical lemmas, which we collect
here for reference.

%
%
\begin{lem}\label{Hilfslemma1} Let $F, G, H, I\in\mathcal{F}$.
Suppose $\overline{F}\asymp\overline{G}$ and $\overline{H}\asymp
\overline{I}$.
Then $\overline{F*H}\asymp\overline{G*I}$.
\end{lem}

{A} proof can be found in \cite{Shimura2005}, Proposition~2.7.
Recall from Section~\ref{ssn:basic} that $\kG$ denotes the set of
nonnegative, unbounded and nondecreasing real functions.

%
%
\begin{lem}
\label{Hilfslemma2} Suppose $g\in\kG$. Then:
\[
\limsup_{x\rightarrow\infty}\frac{\wP(S_{2}>x,X_{1}\wedge
X_{2}>g(x))}{\wP(\hat{S}_{2}>x,Y_{1}\wedge Y_{2}>g(x))}\leq \biggl(\limsup
_{x\rightarrow\infty}\frac{\overline{F}(x)}{\overline{G}(x)} \biggr)^{2}.
\]
\end{lem}

{A} proof can be found in \cite{Foss2011}, Lemma~2.36.

%
%
\begin{lem}\label{Hilfslemma3}
For each $F\in\kF$, $c\geq0$ and $n\ge2$, we have
\[
\lim_{K\rightarrow\infty}\limsup_{x\rightarrow\infty} \wP
(X_{1,n}>x-c,X_{2,n}>K|S_{n}>x)=0.
\]
\end{lem}

\begin{pf}
For $x-c\ge K\ge c$, we have
\begin{eqnarray*}
&&\wP(X_{1,n}> x-c,X_{2,n}>K|S_{n}>x)
\\
&&\quad\le\frac{\wP(X_{1,n}>x-c,X_{2,n}>K)}{\wP
(X_{1,n}>x-c,X_{2,n}>c)}=\frac{1-(1-\overline{F}(x-c))^{n}-n\overline
{F}(x-c)(F(K))^{n-1}}{1-(1-\overline{F}(x-c))^{n}-n\overline
{F}(x-c)(F(c))^{n-1}}
\\
&&\quad= \frac{n\overline{F}(x-c)(1-F(K)^{n-1})+\mathrm
{o}(\overline
{F}(x-c))}{n\overline{F}(x-c)(1-F(c)^{n-1})+\mathrm{o}(\overline
{F}(x-c))}=\frac
{1-F(K)^{n-1}+\mathrm{o}(1)}{1-F(c)^{n-1}+\mathrm{o}(1)},
\end{eqnarray*}
and the result follows by passing to the limit.
\end{pf}

%
%
\begin{lem}\label{Hilfslemma4}
Let $\gamma\geq0$ and $F\in\mathcal{L}(\gamma)$.
Then $F\in\mathcal{S}(\gamma)$ if and only if
%
%
\begin{equation}
\label{eq:inlemma28} \mathbb{P}\bigl(X_{1}+X_{2}>x,
\min(X_{1},X_{2})>h(x)\bigr)=\mathrm {o}\bigl(
\overline{F}(x)\bigr) \qquad\mbox{as } x\rightarrow\infty
\end{equation}
for all $h\in\mathcal{G}$.
\end{lem}

The case $\gamma=0$ is shown in \cite{Asmussen2003}, Proposition~2, and
the case $\gamma>0$ is analogous.
In some proofs, we will need the dominated convergence theorem and
for its application an upper bound for $\overline{F^{n\ast
}}(x)/\overline{F}(x)$
is required. One such is given by the lemma below, known as Kesten's
lemma.

%
%
\begin{lem}
\label{la:kesten} If $F\in\kOS$ then, for every $\varepsilon>0$,
there exists $c>0$ such that for all $n\geq1$ and\vspace*{-1pt} $x\geq0$:
\[
\frac{\overline{F^{n\ast}}(x)}{
\overline{F}(x)} \leq c(c_{F}+\varepsilon-1)^{n}.
\]
\end{lem}

A proof can be found in \cite{Shimura2005}, Proposition~2.4.

\subsection{Proof of Proposition \texorpdfstring{\protect\ref{lem:basic-properties}}{3}}\vspace*{-8pt}
\label{ssn:basic}

\begin{pf*}{Proof of Proposition~\ref{lem:basic-properties}}
(a) We show $\kB^{(n)}=\kB_{1}^{(n)}=\kB_{2}^{(n)}=\kB
_{3}^{(n)}=\kB_{4}^{(n)}$.
The inclusions $\kB^{(n)}\supseteq\kB_{1}^{(n)}\supseteq\mathcal
{J}_{2}^{(n)}$ and $\kB_{4}^{(n)}\subseteq\mathcal{J}_{3}^{(n)}$
are obvious. The inclusion $\mathcal{J}_{1}^{(n)}\subseteq\kB_{4}^{(n)}$
follows immediately from\vspace*{-1pt}
\[
\biggl\{X_{2,n}<\frac{K}{n-1},S_{n}>x \biggr\}\subseteq
\{ X_{1,n}>S_n-K,S_{n}>x\},
\]
for all $F\in\kF$, $K>0$ and $x>0$. It remains to show $\mathcal
{J}^{(n)}\subseteq\mathcal{J}_{1}^{(n)}\subseteq\mathcal{J}_{2}^{(n)}$
and $\mathcal{J}_{1}^{(n)}\supseteq\kB_{3}^{(n)}$.

First, we prove the inclusion $\mathcal{J}^{(n)}\subseteq\mathcal
{J}_{1}^{(n)}$.
Suppose $F\in\kB^{(n)}$ and $F\notin\mathcal{J}_{1}^{(n)}$, then there
is $\tau>0$ such
that for any $m\geq1$:\vspace*{-1pt}
\[
\liminf_{x\rightarrow\infty} \wP(X_{2,n}\leq m|S_{n}>x)
\leq 1-\tau.
\]
For every $m\geq1$, we choose an unbounded and strictly increasing
sequence $(x_{k}^{m})_{k\in\mathbb{N}}$ with $\lim_{m\rightarrow
\infty}
x_{m}^{m}=\infty$
such that for all $k\in\mathbb{N}$:\vspace*{-1pt}
\[
\wP\bigl(X_{2,n}\leq m|S_{n}>x_{k}^{m}
\bigr)\leq1-\frac{\tau}{2}.
\]
Hence, we obtain\vspace*{-1pt}
\[
\limsup_{m\rightarrow\infty} \wP\bigl(X_{2,n}\leq
m|S_{n}>x_{m}^{m}\bigr)\leq1-\frac{\tau}{2},
\]
contradicting the fact that\vspace*{-1pt}
\[
\lim_{x\rightarrow\infty}\mathbb {P}\bigl(X_{2,n}>g(x)
|S_{n}>x\bigr)=0\qquad\mbox{for all }g\in\mathcal{G},
\]
so $\kB^{(n)}\subseteq\kB_{1}^{(n)}$.

Next, we show the inclusion $\mathcal{J}_{1}^{(n)}\subseteq\mathcal
{J}_{2}^{(n)}$.
Suppose $F\in\kB_{1}^{(n)}$.
By definition we know that for every $\varepsilon>0$ there are constants
$x_{0}$ and $K_{0}$ such that for all $x\geq x_{0}$ and $K\geq K_{0}$:
%
%
\begin{equation}
\label{eq:B_2_eq1} \wP(X_{2,n}\leq K|S_{n}>x)\geq1-\varepsilon.
\end{equation}
Let $\delta>0$. Increasing $K_{0}$ if necessary, we can assume that
$\wP(X_{2,n}\leq K)\geq1-\delta$ for all $K\geq K_{0}$. Hence, we
obtain for $x\leq x_{0}$ and $K\geq K_{0}$:
%
%
\begin{equation}
\label{eq:B_2_eq2} \wP(X_{2,n}\leq K|S_{n}>x) \geq
\frac{\wP(X_{2,n}\leq K)+\wP
(S_{n}>x)-1}{\wP(S_{n}>x)}\geq1-\frac{\delta}{\wP(S_{n}>x_{0})}.
\end{equation}
By (\ref{eq:B_2_eq1}) and (\ref{eq:B_2_eq2}) we see that $F\in\kB
_{2}^{(n)}$,
since $\delta>0$ and $\varepsilon>0$ are arbitrary.

Finally, we show $\mathcal{J}_{1}^{(n)}\supseteq\kB_{3}^{(n)}$. Suppose
$F\in\kB_{3}^{(n)}$ and $F\notin\kB_{1}^{(n)}$. Then there exists
some $\delta>0$ such that for any $m\geq1$:
\[
\liminf_{x \to\infty} \wP(X_{2,n}\leq m|S_{n}>x)
\leq1-2\delta.
\]
For every $m\geq1$, we choose an unbounded and strictly increasing
sequence $(x_{k}^{m})_{k\in\mathbb{N}}$
such that for all $k\in\mathbb{N}$:
\[
\wP\bigl(X_{2,n}\leq m|S_{n}>x_{k}^{m}
\bigr)\leq1-\delta.
\]
Since $F\in\kB_{3}^{(n)}$ we know there exist $c>0$ and $\bar{x}>0$
such that for all $x\geq\bar{x}$:
\[
\wP(X_{1,n}>x-c|S_{n}>x)\geq1-\frac{\delta}{3}.
\]
Hence, we obtain for any $m\geq1$ and all $k\geq1$, $x_{k}^{m}\geq
\bar{x}$:
\begin{eqnarray*}
&&\wP\bigl(X_{1,n}> x_{k}^{m}-c,
X_{2,n}>m|S_{n}>x_{k}^{m}\bigr)
\\
&&\quad= \wP\bigl(X_{1,n}>x_{k}^{m}-c|S_{n}>x_{k}^{m}
\bigr)-\wP \bigl(X_{1,n}>x_{k}^{m}-c,X_{2,n}
\leq m|S_{n}>x_{k}^{m}\bigr)
\\
&&\quad\geq\wP\bigl(X_{1,n}>x_{k}^{m}-c|S_{n}>x_{k}^{m}
\bigr)-\wP \bigl(X_{2,n}\leq m|S_{n}>x_{k}^{m}
\bigr)
\\
&&\quad\geq1-\frac{\delta}{3}-1+\delta=\frac{2}{3}\delta.
\end{eqnarray*}
We get
\[
\lim_{m\rightarrow\infty}\limsup_{x\rightarrow\infty } \wP
(X_{1,n}>x-c, X_{2,n}>m|S_{n}>x)>0,
\]
which contradicts Lemma~\ref{Hilfslemma3}.

(b) We show $\mathcal{J}^{(n+1)}=\mathcal{J}^{(n)}$. It suffices to prove
the inclusion $\mathcal{J}_{2}^{(n+1)}\supseteq\mathcal{J}_{2}^{(n)}$.
Then, we can conclude from (a) that $\mathcal{J}^{(n+1)}\supseteq
\mathcal
{J}^{(n)}$.

Suppose $F\in\kB_{2}^{(n)}$. By definition of $\kB_{2}^{(n)}$ we
know for all $\varepsilon>0$ there exists a constant $K_{0}>0$ such
that for all $x\geq0$:
\[
\wP(X_{2,n}\leq K_{0},S_{n}>x)\geq \biggl(1-
\frac{\varepsilon
}{n+1} \biggr)\overline{F^{n\ast}}(x).
\]
Hence, we obtain for $K\geq K_{0}$ and $x\geq0$:
\begin{eqnarray*}
\wP(X_{2,(1,\ldots,n)}\leq K|S_{n+1}>x) & = & \frac{\int^{\infty
}_{0}\wP(X_{2,n}\leq K,S_{n}>x-t) \,\mathrm
{d}F(t)}{\overline
{F^{(n+1)\ast}}(x)}
\\
& \geq& \frac{(1-\frace{\varepsilon}{n+1})\int^{\infty
}_{0}\overline{F^{n\ast}}(x-t)\,\mathrm{d}F(t)}{\overline
{F^{(n+1)\ast}}(x)}=1-\frac{\varepsilon}{n+1}.
\end{eqnarray*}
Thus, we see that for all $K\geq K_{0}$ and $x\geq0$:
\begin{eqnarray*}
\wP(X_{2,n+1}\leq K|S_{n+1}>x) & = & \wP (X_{2,(1,\ldots,n)}\leq
K,X_{2,(1,\ldots,n-1,n+1)}\leq K,
\\
&&\phantom{\wP (} \ldots,X_{2,(2,\ldots,n,n+1)}\leq K|S_{n+1}>x )
\\
& \geq& (n+1)\wP(X_{2,n}\leq K|S_{n+1}>x)-n
\\
& \geq&1-\varepsilon,
\end{eqnarray*}
where we used the inequality
\[
\wP \Biggl(\bigcap_{i=1}^{n+1}A_{i}
\Biggr)\geq\sum^{n+1}_{i=1}
\wP(A_{i})-n.
\]
We obtain $F\in\kB_{2}^{(n+1)}$.
\end{pf*}

In the second part of the proof of Proposition~\ref
{lem:basic-properties}(b) we will use Proposition~\ref{lem:Class_B}(a),
for that reason we give
the proof of Proposition~\ref{lem:Class_B}(a) already here.

\begin{pf*}{Proof of Proposition~\ref{lem:Class_B}(a)}
We prove $\mathcal{J}^{(n)}\subseteq\mathcal{OS}$. Let $n\geq3$.
Suppose that $F\in\mathcal{J}^{(n)}=\mathcal{J}_{1}^{(n)}$. Then,
\begin{eqnarray*}
1 & = & \lim_{K\rightarrow\infty}\liminf_{x\rightarrow\infty}\wP
\biggl(X_{2,n}<\frac{K}{n-1} \Big|S_{n}>x \biggr)
\\
& \leq& \lim_{K\rightarrow\infty}\liminf_{x\rightarrow\infty} \wP
(X_{1,n}>x|S_{n}>x+K)
\\
& \leq& \lim_{K\rightarrow\infty}\liminf_{x\rightarrow\infty
}
\frac
{n\wP(X_{n}>x)}{\wP(S_{n-1}>x)\wP(X_{n}>K)}
\\
& \leq& n\lim_{K\rightarrow\infty} \biggl(\frac
{1}{\wP(X_{n}>K)} \liminf
_{x\rightarrow\infty}\frac{\wP
(X_{n}>x)}{\wP
(S_{n-1}>x)} \biggr).
\end{eqnarray*}
Hence, we have $\liminf_{x\rightarrow\infty}\frac{\wP
(X_{1}>x)}{\wP
(S_{n-1}>x)}>0$
and thus $\limsup_{x\rightarrow\infty}\frac{\wP(S_{n-1}>x)}{\wP
(X_{1}>x)}<\infty$.

In the case $n=2$ we use the inclusion $\mathcal{J}^{(3)}\supseteq
\mathcal{J}^{(2)}$,
which was already shown above, to get $F\in\mathcal{J}^{(3)}$.
\end{pf*}

We now resume the second part of the proof of Proposition~\ref
{lem:basic-properties}.

\begin{pf*}{Proof of Proposition~\ref{lem:basic-properties}}
We begin with the inclusion $\mathcal{J}^{(n+1)}\subseteq\mathcal
{J}^{n}$.
Suppose $F\in\mathcal{J}^{(n+1)}$ and $F\notin\mathcal{J}^{(n)}$.
Then there exists $g\in\mathcal{G}$ such that
\[
\limsup_{x\rightarrow\infty} \wP\bigl(X_{2,n}>g(x)|S_{n}>x
\bigr)>0
\]
and
\[
\lim_{x\rightarrow\infty}\wP\bigl(X_{2,n+1}>g(x)|S_{n+1}>x
\bigr)=0.
\]
Thus, we have:
%
%
\begin{equation}
\label{eq:WIder2-2} \limsup_{x\rightarrow\infty}\frac{\wP(X_{2,n}>g(x),S_{n}>x)}{\wP
(X_{2,n+1}>g(x),S_{n+1}>x)}
\frac{\wP(S_{n+1}>x)}{\wP(S_{n}>x)}\leq \limsup_{x\rightarrow\infty}\frac{\wP(S_{n+1}>x)}{\wP
(S_{n}>x)}=\infty .
\end{equation}

By Proposition~\ref{lem:Class_B}(a), we obtain $F\in\kB
^{(n+1)}\Rightarrow F\in\kOS$.
From the identical convolution closure of $\kOS$
(see also \cite{Shimura2005},
page 452, Proposition~2.5(iv)), we see that $F\in\kOS\Rightarrow
F^{n\ast}\in\kOS\Leftrightarrow\limsup_{x\rightarrow\infty}\frac{\wP(S_{2n}>x)}{\wP(S_{n}>x)}<\infty$,
by (\ref{eq:WIder2-2}) we obtain a contradiction.
\end{pf*}

\subsection{Proof of Proposition \texorpdfstring{\protect\ref{lem:Class_B}}{5}}\vspace*{-9pt}
\label{ssn:classi}

\begin{pf*}{Proof of Proposition~\ref{lem:Class_B}}
(a) This was already shown above as part of the proof of Proposition~\ref
{lem:basic-properties}.

(b) and (b$'$) Let $\gamma\geq0$. If $F \in\mathcal{S}(\gamma)$,
then $F
\in\mathcal{L}(\gamma)$ and, by Lemma~\ref{Hilfslemma4}, $\F\in
\mathcal{J}$, so
$\mathcal{S}(\gamma) \subseteq\mathcal{J}\cap\mathcal{L}(\gamma)$.

Conversely, let $F \in\mathcal{J}\cap\mathcal{L}(\gamma)$. From $F
\in\mathcal{J} \subset\kOS$, it follows that $F$ satisfies
(\ref{eq:inlemma28}) and therefore Lemma~\ref{Hilfslemma4} implies that $F \in\mathcal{S}(\gamma)$.

%

(c) We show $\mathcal{D}\subseteq\kB$. Let $F\in\mathcal{D}$ and
\[
\gamma:=\sup_{x\ge0}\frac{\overline{F}(x/2)}{\overline{F}(x)}.
\]
Let $\varepsilon>0$. There exists $K_{0}>0$ such that for all $K\ge K_{0}$
\[
\wP(X_{2}>K)\gamma<\frac{\varepsilon}{2}.
\]
For $K\ge K_{0}$ and $x$ such that $x\ge2K$, we get
\begin{eqnarray*}
\wP(X_{1}\wedge X_{2}>K|S_{2}>x) & \le& 2\wP
\biggl(X_{1}>\frac
{x}{2},X_{2}>K|S_{2}>x
\biggr)
\\
& \le& 2\wP \biggl(X_{1}>\frac{x}{2} \biggr)\frac{\wP(X_{2}>K)}{\wP
(S_{2}>x)}
\\
& \leq& 2\wP(X_{2}>K)\gamma<\varepsilon.
\end{eqnarray*}
Since $\varepsilon>0$ was arbitrary, the assertion follows.
\end{pf*}

\subsection{Proof of Proposition \texorpdfstring{\protect\ref{TailClosure-Lemma}}{8}}
\label{ssn:tailclose}

Next we prove Proposition~\ref{TailClosure-Lemma},
which establishes tail closure property of $\kB$.

\begin{pf*}{Proof of Proposition~\ref{TailClosure-Lemma}}
Suppose $F\in\kB$, $\overline{F}\asymp\overline{G}$ and $G\notin
\kB$.
There exists $h\in\mathcal{G}$ such that
\[
\limsup_{x\rightarrow\infty}\wP\bigl(\mathrm{Y}_{2,2}>h(x)|
\hat{S}_{2}>x\bigr)>0.
\]
Thus, we have by the definition of $\kB$:
\[
\limsup_{x\rightarrow\infty}\frac{\wP(\mathrm{Y}_{2,2}>h(x),\hat
{S}_{2}>x)}{\wP(\hat{S}_{2}>x)}\frac{\wP(S_{2}>x)}{\wP
(X_{2,2}>h(x),S_{2}>x)}=\infty.
\]
By Lemmas~\ref{Hilfslemma1} and~\ref{Hilfslemma2}, we get a contradiction.
\end{pf*}

\subsection{Proof of Proposition \texorpdfstring{\protect\ref{Convolution-Lemma}}{10}}
\label{ssn:closeproof}

We prove the convolution closure properties of the class $\kB$.

\begin{pf*}{Proof of Proposition~\ref{Convolution-Lemma}}
(a) We prove closure under convolution powers of $\kB$,
that is, if $F\in\mathcal{J}$ then $F^{n\stern}\in\kB$. Suppose
$F^{n\stern}\in\mathcal{J}$.
We show $\overline{F^{n\stern}}\asymp\overline{F^{(n+1)\stern}}$
and hence ${F^{(n+1)\stern}}\in\kB$. From $S_{n}\in\kB\subset\kOS$
we obtain
\[
\limsup_{x\rightarrow\infty}\frac{\wP(S_{n+1}>x)}{\wP
(S_{n}>x)}\frac{\wP
(S_{2n}>x)}{\wP(S_{2n}>x)}\leq
c_{F^{n\stern}}\limsup_{x\rightarrow
\infty}\frac{\wP(S_{n+1}>x)}{\wP(S_{2n}>x)}\leq
c_{F^{n\stern}}.
\]

(b) We prove closure under convolution for tail-equivalent random variables
from the class $\kB$, that is, if $F\in\mathcal{J}$ and $\overline
{F}\asymp\overline{G}$,
then $F\stern G\in\kB$. Suppose $F\in\kB$, $\overline{F}\asymp
\overline{G}$
and $F*G\notin\kB$. Then, there exists an $h\in\mathcal{G}$ such that
%
%
\begin{equation}
\label{eq:Mark1-1} \limsup_{x\rightarrow\infty} \wP \bigl((X_{1}+Y_{1})
\wedge (X_{2}+Y_{2})>h(x) |\hat{S}_{2}+S_{2}>x
\bigr)>0.
\end{equation}
By $F\in\kB$ and (a) we have that $F^{2\ast}\in\kB$
and it follows by definition that
%
%
\begin{equation}
\label{eq:Mark2-1} \limsup_{x\rightarrow\infty} \wP \bigl((X_{1}+X_{2})
\wedge (X_{3}+X_{4})>h(x) |S_{4}>x \bigr)=0.
\end{equation}
Combining (\ref{eq:Mark1-1}) and (\ref{eq:Mark2-1}) yields
%
%
\begin{equation}
\label{eq:Wider3-1} \limsup_{x\rightarrow\infty}\frac{\wP((X_{1}+Y_{1})\wedge
(X_{2}+Y_{2})>h(x),\hat{S}_{2}+S_{2}>x)}{\wP((X_{1}+X_{2})\wedge
(X_{3}+X_{4})>h(x),S_{4}>x)}
\frac{\wP(S_{4}>x)}{\wP(\hat
{S}_{2}+S_{2}>x)}=\infty.
\end{equation}
By Lemma~\ref{Hilfslemma1}, we obtain $\overline{(F^{2\ast
})*(G^{2\ast})}\asymp
\overline{F^{4\ast}}$,
that is, $\limsup_{x\rightarrow\infty}\frac{\wP(S_{4}>x)}{\wP
(\hat
{S}_{2}+S_{2}>x)}<\infty$.
Hence, by Lemma~\ref{Hilfslemma2} and (\ref{eq:Wider3-1}) we get
a contradiction.

(c) We show root convolution closure {for} $\kB$,
that is, if $F^{n\stern}\in\kB$ then $F\in\kB$. Let $n=2^{m}$,
$m\in
\mathbb{N}$.
Suppose $F^{2^{m}\ast}\in\kB$.
Since $\kB\subset\kOS$ we have $F^{2^{m}\stern}\in\kOS$
and hence there exists a constant $c_{2^{m}}$ such that
\[
\liminf_{x\rightarrow\infty}\frac{\wP(S_{2^{m-1}}>x)}{\wP
(S_{2^{m}}>x)}>c_{2^{m}}>0.
\]
We obtain by definition for all $h\in\mathcal{G}$
%
%
\begin{eqnarray}
\label{eq:step-1} 0 & = & \limsup_{x\rightarrow\infty}\wP \bigl(S_{2^{m}}^{(1)}
\wedge S_{2^{m}}^{(2)}>h(x) |{S_{2^{m}}^{(1)}+S_{2^{m}}^{(2)}}>x
\bigr)
\nonumber
\\[-8pt]
\\[-8pt]
& \geq& c_{2^{m}}\limsup_{x\rightarrow\infty}\wP
\bigl(S_{2^{m-1}}^{(1)}\wedge S_{2^{m-1}}^{(2)}>h(x)
|{S_{2^{m-1}}^{(1)}+S_{2^{m-1}}^{(2)}}>x \bigr).
\nonumber
\end{eqnarray}
Thus, we have $F^{2^{m-1}\stern}\in\kB$. We repeat
the argument leading to (\ref{eq:step-1}) for
$(m-1)$ times and arrive at
%
%
\begin{eqnarray}
\label{eq:Rootmark-1} 0 & \ge& {c_{2^{m}}\limsup_{x\rightarrow\infty}\wP
\bigl(S_{2^{m-1}}^{(1)}\wedge S_{2^{m-1}}^{(2)}>h(x)
|S_{2^{m-1}}^{(1)}+S_{2^{m-1}}^{(2)}>x \bigr)}
\nonumber
\\
& \ge& {c_{2^{m}}\cdot c_{2^{m-1}}\limsup_{x\rightarrow\infty}
\wP \bigl(S_{2^{m-2}}^{(1)}\wedge S_{2^{m-2}}^{(2)}>h(x)
|S_{2^{m-2}}^{(1)}+S_{2^{m-2}}^{(2)}>x \bigr)}
\nonumber
\\[-8pt]
\\[-8pt]
& \vdots&
\nonumber
\\
& \geq& {c_{2^{m}}\cdots c_{2}}\limsup_{x\rightarrow\infty}
\wP \bigl(X_{1}\wedge X_{2}>h(x)|S_{2}>x \bigr),
\nonumber
\end{eqnarray}
which gives $F\in\kB$. In case $n\neq2^{m}$ for
all $m\in\mathbb{N}$, we take $\widetilde{m}:=\min\{m\in\mathbb
{N}:n<2^{m}\}$.
Denote by $k:=2^{\widetilde{m}}$. By the argument
in the proof of (a),
we know that $F^{n}\in\kB\Rightarrow F^{k}\in\kB$.
From (\ref{eq:Rootmark-1}), we obtain $F\in\kB$.
\end{pf*}

\subsection{Proof of Proposition \texorpdfstring{\protect\ref{lem:mix}}{12}}\vspace*{-9pt}
\label{ssn:proofmixture}

\begin{pf*}{Proof of Proposition~\ref{lem:mix}}
%
(a) The equivalence (i)$\,\Leftrightarrow\,$(iii) follows from
(\ref{eq:mixequiv}) and Lemma~\ref{TailClosure-Lemma}. Next, we show the
equivalence (i)$\,\Leftrightarrow\,$(ii). Let $(X+Y)\in\mathcal{J}$
(with our usual slight abuse of notation). To show
that $(X\vee Y)\in\mathcal{J}$, abbreviate $V_{i}:=X_{i}\vee Y_{i}$
for $i\in\{1,2,3,4\}$. Then, for every $g\in\mathcal{G}$ and $x\ge0$,
%
%
\begin{eqnarray}
\label{eq:Mark0} \frac{\wP(V_{1}\wedge V_{2}>g(x),{V_{1}+V_{2}}>x)}{\wP(V_{1}+V_{2}>x)} & \leq& \frac{\wP(X_{1}\wedge X_{2}>g(x),X_{1}+X_{2}>x)}{\wP
(V_{1}+V_{2}>x)}
\nonumber
\\
& &{} +\frac{\wP(Y_{1}\wedge Y_{2}>g(x),Y_{1}+Y_{2}>x)}{\wP
(V_{1}+V_{2}>x)}
\\
& &{} +2\frac{\wP(X_{1}\wedge Y_{2}>g(x),X_{1}+Y_{2}>x)}{\wP
(V_{1}+V_{2}>x)}.
\nonumber
\end{eqnarray}
From $X\in\kB$, we obtain for the first term on the right-hand side
of (\ref{eq:Mark0}):
\[
\limsup_{x\rightarrow\infty}\frac{\wP(X_{1}\wedge
X_{2}>g(x),S_{2}>x)}{\wP(V_{1}+V_{2}>x)} \leq\limsup
_{x\rightarrow
\infty}\frac{\wP(X_{1}\wedge X_{2}>g(x),S_{2}>x)}{\wP(S_{2}>x)}=0.
\]
Analogously for the second term:
\[
\limsup_{x\rightarrow\infty}\frac{\wP(Y_{1}\wedge
Y_{2}>g(x),Y_{1}+Y_{2}>x)}{\wP(V_{1}+V_{2}>x)}=0.
\]
From $(X+Y)\in\mathcal{J}$, we obtain for the third term on the right-hand
side of (\ref{eq:Mark0}):
\begin{eqnarray*}
&&\limsup_{x\rightarrow\infty} 2 \frac{\wP(X_{1}\wedge
Y_{2}>g(x),X_{1}+Y_{2}>x)}{\wP(V_{1}+V_{2}>x)}
\\
&&\quad\le2\limsup_{x\rightarrow\infty} \biggl(\frac{\wP
(X_{1}\wedge
Y_{2}>g(x),S_{2}+\hat{S}_{2}>x)}{\wP(S_{2}+\hat{S}_{2}>x)}
\frac{\wP
(S_{2}+\hat{S}_{2}>x)}{\wP(X_{1}+Y_{1}>x)} \biggr)=0.
\end{eqnarray*}
Altogether, we arrive at
\begin{eqnarray*}
&&\limsup_{x\rightarrow\infty} \wP \bigl((X_{1}\vee
Y_{1})\wedge (X_{2}\vee Y_{2})>g(x) |
(X_{1}\vee Y_{1})+(X_{2}\vee Y_{2})>x
\bigr)
\\
&&\quad=\limsup_{x\rightarrow\infty}\frac{\wP(V_{1}\wedge
V_{2}>g(x),V_{1}+V_{2}>x)}{\wP(V_{1}+V_{2}>x)}=0
\end{eqnarray*}
for all $g\in\mathcal{G}$, that is, by (\ref{defnBwithG}),
$(X\vee Y)\in\mathcal{J}$.

For the opposite implication $(X\vee Y)\in\mathcal{J}\Rightarrow
(X+Y)\in\mathcal{J}$
abbreviate $W_{i}:=X_{i}+Y_{i}$ for $i\in\{1,2\}$.
From $(X\vee Y)\in\kB$ and ${V_{1}+V_{2}}\in\kB\subset\kOS$, we obtain
for all $g\in\mathcal{G}$
\begin{eqnarray*}
&&\limsup_{x\rightarrow\infty} \frac{\wP(W_{1}\wedge
W_{2}>g(x),{W_{1}+W_{2}}>x)}{\wP({W_{1}+W_{2}}>x)}
\\
&&\quad\leq\limsup_{x\rightarrow\infty} \biggl(\frac{\wP
(({V_{1}+V_{2}})\wedge
({V_{3}+V_{4}})>g(x),{V_{1}+\cdots+V_{4}}>x)}{\wP({V_{1}+\cdots
+V_{4}}>x)}
\frac{\wP({V_{1}+\cdots+V_{4}}>x)}{\wP
({V_{1}+V_{2}}>x)} \biggr)
\\
&&\quad=0,
\end{eqnarray*}
hence by (\ref{defnBwithG}) $(X+Y)\in\mathcal{J}$,
which completes the proof of the equivalence (i)$\,\Leftrightarrow
\,$(ii).

Next, we prove Proposition~\ref{lem:mix}(b). The proof is analogous
to the proof of the same assertion for the class $\kOS$, see
\cite{Lin}, Lemma~3.1.
Let $L_{i}:=X_{i}\wedge Y_{i}$ for $i \in\{1,2\}$. For all $g\in
\mathcal{G}$ we have
\begin{eqnarray*}
&& \frac{\wP(L_{1}\wedge L_{2}>g(x),L_{1}+L_{2}>x)}{\wP
(L_{1}+L_{2}>x)}
\\
&&\quad\leq\frac{\int_{g(x)}^{\infty}\wP
(X_{1}>(x-y)\vee g(x))\wP(Y_{1}>(x-y)\vee g(x))\,\mathrm
{d}F_{L_{2}}(y)}{\wP
(L_{1}>x)}
\\
&&\quad\leq\frac{\int_{g(x)}^{\infty}\wP
(X_{1}>(x-y)\vee g(x))\wP(Y_{1}>(x-y)\vee g(x))\wP(Y_{2}\ge
y)\,\mathrm{d}F_{X_{2}}(y)}{\wP(X_{1}>x)\wP(Y_{1}>x)}
\\
&&\qquad{} +\frac{\int_{g(x)}^{\infty}\wP
(X_{1}>(x-y)\vee g(x))\wP(Y_{1}>(x-y)\vee g(x))\wP(X_{2} \ge
y)\,\mathrm{d}F_{Y_{2}}(y)}{\wP(X_{1}>x)\wP(Y_{1}>x)}.
\end{eqnarray*}
Using the {inequality}
\[
\wP\bigl(Y_{1}>(x-y)\vee g(x)\bigr)\wP(Y_{2}\ge y)\leq
\wP(Y_{1}+Y_{2}>x) %
\]
we obtain
\begin{eqnarray*}
&& \limsup_{x\rightarrow\infty} \frac{\wP(L_{1}\wedge
L_{2}>g(x),L_{1}+L_{2}>x)}{\wP(L_{1}+L_{2}>x)}
\\
&&\quad\leq\limsup_{x\rightarrow\infty}\frac{\wP
(Y_{1}+Y_{2}>x)}{\wP
(Y_{1}>x)}\frac{\int_{g(x)}^{\infty}\wP
(X_{1}>(x-y)\vee g(x))\,\mathrm{d}F_{X_{2}}(y)}{\wP(X_{1}>x)}
\\
&&\qquad{} +\limsup_{x\rightarrow\infty}\frac{\wP
(X_{1}+X_{2}>x)}{\wP
(X_{1}>x)}\frac{\int_{g(x)}^{\infty}\wP
(Y_{1}>(x-y)\vee g(x))\,\mathrm{d}F_{Y_{2}}(y)}{\wP(Y_{1}>x)}
\\
&&\quad= \limsup_{x\rightarrow\infty}\frac{\wP(\hat
{S}_{2}>x)}{\wP
(Y_{1}>x)}\frac{\wP(S_{2}>x,X_{1}\wedge X_{2}>g(x))}{\wP(X_{1}>x)}
\\
&& \qquad{} +\limsup_{x\rightarrow\infty}\frac{\wP(S_{2}>x)}{\wP
(X_{1}>x)}\frac
{\wP(\hat{S}_{2}>x,Y_{1}\wedge Y_{2}>g(x))}{\wP(Y_{1}>x)}
\\
&&\quad= 0,
\end{eqnarray*}
since $F,G \in\kB\subset\kOS$. The proof is complete.
\end{pf*}

\subsection{Proofs of Propositions \texorpdfstring{\protect\ref{lem:Zufaellige-Sum-B}}{14}
and \texorpdfstring{\protect\ref{lem:Zufaellige-Sum-B-2}}{16}}\label{ssn:proofrandsum}

We begin with
Proposition~\ref{lem:Zufaellige-Sum-B}.

\begin{pf*}{Proofs of Proposition~\ref{lem:Zufaellige-Sum-B}}
(a) Suppose $F\in\mathcal{J}$ and $\sum^{\infty
}_{k=1}p_{k}({c_{F}+\varepsilon-1})^{k}<\infty$ for some $\varepsilon>0$.
Recall that we need to show that $\overline{F}_{N}\asymp\overline{F}$.
From Lemma~\ref{la:kesten} (Kesten's) and $F\in\kB\subset\mathcal{OS}$
we obtain for some suitable $c_{1}\in(0,\infty)$ and all $x\geq0$,
\begin{eqnarray*}
\overline{F}_{N}(x)&=&{\sum_{k=1}^{\infty}}p_{k}
\overline{F^{k\ast
}}(x)\leq{\sum_{k=1}^{\infty}}
{p_{k}c_{1}(c_{F}+\varepsilon
-1)^{k}}\overline{F}(x)
\\
&=&\overline{F}(x){\sum
_{k=1}^{\infty
}}p_{k}{c_{1}(c_{F}+
\varepsilon-1)^{k}}.
\end{eqnarray*}
Hence, we see that $\limsup_{x\rightarrow\infty}\overline
{F_{N}}(x)/\overline{F}(x)<\infty$.
For the lower bound, pick some $k \ge1$ with $p_{k}>0$. Then, for
all $x\geq0$,
\[
\overline{F}_{N}(x)\geq p_{k}\wP(S_{k}>x)\geq
p_{k}\overline{F}(x).
\]

We obtain $\overline{F}_{N}\asymp\overline{F}$ and therefore
$F_{N}\in
\mathcal{J}$.

(b) Suppose $F_{N}\in\mathcal{J}$ and that $\sum^{\infty}_{k=1}p_{k}
({c_{F_{N}}+\varepsilon-1})^{k}<\infty$ for some $\varepsilon>0$.
Again, we need to prove that $\overline{F_{N}}\asymp\overline{F}$.
To this end, by means of contradiction, suppose that for every integer
$n\geq2$,
\[
\liminf_{x\rightarrow\infty} \frac{\overline{F^{n\ast}}(x)}{\overline{F_{N}}(x)}=0.
\]
Our proof then splits into two cases:

\textit{Case 1}: $p_{0}=0$.
For every $n\geq1$, we choose an unbounded and strictly increasing
sequence $(x_{k}^{n})_{k\in\mathbb{N}}$ such that for all $n\in
\mathbb{N}$
\[
\lim_{m\rightarrow\infty}\frac{\overline{F^{n\ast
}}(x_{m}^{n})}{\overline{F_{N}}(x_{m}^{n})}=0 \quad\mbox{and in particular}
\quad \lim_{m\rightarrow\infty}\frac{\overline{F^{m\ast
}}(x_{m}^{m})}{\overline{F_{N}}(x_{m}^{m})}=0.
\]
From Lemma~\ref{la:kesten} (Kesten's) and $p_0=0$ we conclude that,
for some suitable $c_2 \in(0, \infty)$, for all $n,m\in\wN$
\[
\frac{\overline{F^{n\ast}}(x_{m}^{n})}{\overline
{F_{N}}(x_{m}^{n})}\leq\frac{\overline{F_{N}^{n\stern
}}(x_{m}^{n})}{\overline{F_{N}}(x_{m}^{n})}\leq{c_2}
({c_{F_N}+\varepsilon-1})^{n}. 
\]
Since by assumption the right-hand side is summable in $n$,
we can use the dominated convergence theorem to arrive at the desired
contradiction:
\begin{eqnarray*}
1&=&\lim_{m\rightarrow\infty}\frac{\overline
{F_{N}}(x_{m}^{m})}{\overline
{F_{N}}(x_{m}^{m})} =\lim_{m\rightarrow\infty}
\sum_{k=1}^{\infty}p_{k}
\frac{\overline
{F^{k\ast}}(x_{m}^{m})}{\overline{F_{N}}(x_{m}^{m})}
\\
&=& \sum_{k=1}^{\infty}p_{k}\lim
_{m\rightarrow\infty}\frac
{\overline{F^{k\ast}}(x_{m}^{m})}{\overline{F_{N}}(x_{m}^{m})}\leq \sum
_{k=1}^{\infty}p_{k}\lim_{m\rightarrow\infty
}
\frac{\overline{F^{m\ast}}(x_{m}^{m})}{\overline{F_{N}}(x_{m}^{m})}=0.
\end{eqnarray*}

\textit{Case 2}: $p_{0}>0$. This can be reduced to {Case 1} by
switching to the reweighted random variable
$\hat{N}$ with probabilities
\[
\hat{p}_{n}:=\wP(\hat{N}=n):=\frac{p_{n}}{1-p_{0}},
\]
for $n>0$ and $\hat{p}_{0}=\wP(\hat{N}=0):=0$. Thanks to {Case
1} we have that $F_{\hat{N}}\in\kB$. Further, observe that
\[
\lim_{x\rightarrow\infty}\frac{\overline{F_{N}}(x)}{\overline
{F_{\hat
{N}}}(x)}=\lim_{x\rightarrow\infty}
\frac{\sum^{\infty
}_{n=0}p_{n}\overline{F^{n\ast}}(x)}{\sum^{\infty}_{n=0}\hat
{p}_{n}\overline{F^{n\ast}}(x)}=\frac{1-p_{0}}{1}.
\]
From Proposition~\ref{TailClosure-Lemma}
and $\overline{F_{N}}\asymp\overline{F_{\hat{N}}}$,
we conclude that $F_{N}\in\kB$.
\end{pf*}

Next, we prove Proposition~\ref{lem:Zufaellige-Sum-B-2} using the
arguments of
the proof of Theorem~1.5 of Watanabe \cite{Watanabe}.

\begin{pf*}{Proof of Proposition~\ref{lem:Zufaellige-Sum-B-2}}
(a) To infer that $\overline{F_{N}}\asymp\overline{F^{m\ast}}$ for
some $m \in\N$, we again argue by contradiction. So suppose that for
every integer $m\geq2$
\[
\liminf_{x\rightarrow\infty}\frac{\overline{F^{m\ast
}}(x)}{\overline{F_{N}}(x)}=0.
\]
From $F_{N}\in\mathcal{OS}$, we know that $c_{F_{N}}<\infty$ and from
our assumption $\liminf_{n\rightarrow\infty}\frac{\wP
(N_{1}+N_{2}>n)}{\wP(N_{1}>n)}>c_{F_{N}}$
we infer that there exists a $\delta>0$ and an integer $m_{0}{
=m_0(\delta)}$ such that,
for every $k\geq m_{0}+1$,
%
%
\begin{equation}
\label{eq:Bed-p-x-p} \frac{\wP(N_{1}+N_{2}>k)}{\wP(N_{1}>k)}>c_{F_{N}}+\delta.
\end{equation}
Let $(x_{n})_{n\in\wN}$ be a strictly increasing sequence with
$\lim_{n\rightarrow\infty}x_{n}=\infty$
such that
%
%
\begin{equation}
\label{eq:bed_zu_null} \lim_{n\rightarrow\infty}\frac{\overline{F^{m_{0}\ast
}}(x_{n})}{\overline{F_{N}}(x_{n})}=0.
\end{equation}
Since $\overline{F^{m_{0}\ast}}(x)\geq\overline{F^{k\ast}}(x)$
for $1\leq k\leq m_{0}$, we have
\[
\lim_{n\rightarrow\infty}\frac{\overline{F^{k\ast
}}(x_{n})}{\overline{F_{N}}(x_{n})}=0
\]
for $1\leq k\leq m_{0}$. As in \cite{Watanabe}, define $I_{j}(n)$
and $J_{j}(n)$ for $j=1,2$ as
\begin{eqnarray*}
J_{1}(n)&=&\sum_{k=0}^{m_{0}}(p*p)_{k}
\overline {F^{k\ast}}(x_{n}), \qquad I_{1}(n)=\sum
_{k=0}^{m_{0}}p_{k}{
\overline{F^{k\ast}}(x_{n}),}
\\
J_{2}(n)&=&\sum_{k=m_{0}+1}^{\infty}(p*p)_{k}
\overline{F^{k\ast
}}(x_{n}), \qquad I_{2}(n)=\sum
_{k=m_{0}+1}^{\infty}p_{k}\overline
{F^{k\ast}}(x_{n}).
\end{eqnarray*}
We see from (\ref{eq:bed_zu_null}) that
%
%
\begin{equation}
\label{eq:I1-J1} \lim_{n\rightarrow\infty}\frac{I_{1}(n)}{\overline
{F_{N}}(x_{n})} = \lim
_{n\rightarrow\infty}\frac
{J_{1}(n)}{\overline{F_{N}}(x_{n})}=0,
\end{equation}
and since $F_{N}^{2\ast}= \sum_{k=0}^{\infty}(p*p)_{k}F^{k\ast}$,
(\ref{eq:Bed-p-x-p}) and (\ref{eq:I1-J1}) give
\[
c_{F_{N}}\geq\limsup_{n\rightarrow\infty}\frac
{\overline{F_{N}^{2\ast}}(x_{n})}{\overline{F_{N}}(x_{n})}=\limsup
_{n\rightarrow\infty}\frac
{(J_{1}(n)+J_{2}(n))/\overline
{F_{N}}(x_{n})}{(I_{1}(n)+I_{2}(n))/\overline{F_{N}}(x_{n})}=\limsup_{n\rightarrow\infty}
\frac{J_{2}(n)}{I_{2}(n)}. 
\]
To arrive at the desired {contradiction}, define
$h_{m_{{0}}+1}(x_{n}):=\overline{F^{(m_{0}+1)\ast}}(x_{n})$
and $h_{j}(x_{n}):=\overline{F^{j\ast}}(x_{n})-\overline
{F^{(j-1)\ast}}(x_{n})$
for $j>m_{0}+1$. We obtain
\begin{eqnarray*}
\limsup_{n\rightarrow\infty}\frac
{J_{2}(n)}{I_{2}(n)} & = & \limsup
_{n\rightarrow\infty}\frac
{\sum_{k=m_{0}+1}^{\infty}(p*p)_{k}\sum^{k}_{j=m_{0}+1}h_{j}(x_{n})}{\sum_{k=m_{0}+1}^{\infty
}p_{k}\sum^{k}_{j=m_{0}+1}h_{j}(x_{n})}
\\
& = & \limsup_{n\rightarrow\infty}\frac{\sum_{j=m_{0}+1}^{\infty
}h_{j}(x_{n})\sum^{\infty}_{{k=j}}\wP(N_{1}+N_{2}=k)}{\sum_{j=m_{0}+1}^{\infty
}h_{j}(x_{n})\sum^{\infty}_{{k=j}}\wP(N_{1}=k)}
\\
& = & \limsup_{n\rightarrow\infty}\frac{\sum_{j=m_{0}+1}^{\infty
}h_{j}(x_{n}){\wP
(N_{1}+N_{2}>j-1)}}{\sum_{j=m_{0}+1}^{\infty}h_{j}(x_{n}){\wP(N_{1}>j-1)}}
\\
& > & c_{F_{N}}+\delta.
\end{eqnarray*}
This is a contradiction. Since $\overline{F^{m\ast}}(x)\leq\overline
{F_{N}}(x)\frac{1}{p_{m}}$
with $p_{m}>0$ for sufficiently large integers $m$, it follows that
$\overline{F_{N}}\asymp\overline{F^{m\ast}}$.

(b) The assertion follows from (a), Proposition~\ref{Convolution-Lemma}
and $\kB\subset\kOS$.
\end{pf*}

\subsection{Proof of Theorem \texorpdfstring{\protect\ref{thm:P-V-E-for-B}}{19}}

We prepare the proof by recalling two results due to Yang and Wang
\cite{Yang2011} for reference.

%
%
\begin{lem}[({\cite{Yang2011}, Lemma~2.2})]\label{lem:YangLemma1}
With the notation of Theorem~\ref{thm:P-V-E},
if $a <0$ and $F_{I}\in\kOL$ then $\overline{G}\asymp\overline{F_{I}}$.
\end{lem}


%
%
\begin{them}[({\cite{Yang2011}, Theorem~1.4})]\label{lem:YangLemma2}
Let $F$
be a distribution function on $(-\infty,\infty)$ such that
$\overline{F}$ is integrable and $F_{I}\in\kOS\cap\mathcal{DK}$.
Further, let $\alpha$ and $\beta$ be two fixed positive constants.
Consider any sequence $\{X_{i}:i\geq1\}$ of independent random variables
such that, for each $i\ge1$, the distribution $F_{i}$
of $X_{i}$ satisfies the conditions
\[
\overline{F_{i}}(x)\leq\overline{F}(x),\qquad\mbox{for all } x\in (-
\infty,\infty)\quad\mbox{and} \quad\intop _{-\infty}^{\infty
}(y\vee-
\beta)\,\mathrm{d}F_{i}(y)\leq-\alpha.
\]
Then there exists a positive constant $r$, depending only on $F$,
$\alpha$ and $\beta$, such that for all sequences $\{X_{i}:i\geq1\}$
as above,
%
\[
\overline{F_{M}}(x)\leq r\overline{F_{I}}(x)
\]
for all $x\in(-\infty,\infty)$.
\end{them}

Now the proof of the P--V--E theorem for class $\kB$ can
simply be reduced to
previously stated results.

\begin{pf*}{Proof of Theorem~\ref{thm:P-V-E-for-B}}
Since by assumption $F_{I}\in\kOL$ and $a:=\mathbb{E}[\bar{X_{k}}]<0$
we obtain from Lemma~\ref{lem:YangLemma1} that $\overline
{F_{I}}\asymp
\overline{G}$.
Hence, the equivalence of (1) and (2) follows from the weak tail-equivalence
closure of the class $\kB$ (Proposition~\ref{TailClosure-Lemma}).

Now additionally assume $p(c_{G}+\varepsilon-1)<1$ holds for some
$\varepsilon>0$ (condition (i)). As we know from (\ref{eq:W_G_relation}),
we can write $F_{M}$ as a random sum
%
%
\begin{equation}
\label{eq:RandomSumFM} \overline{F_{M}}=(1-p)\sum
_{n=0}^{\infty}p^{n}\overline
{G^{n*}}(x).
\end{equation}
Hence, we obtain $\overline{F_{M}}\asymp\overline{G}$ by application
of Proposition~\ref{lem:Zufaellige-Sum-B}. Now, applying the weak
tail-equivalence closure of the class $\kB$ we conclude the equivalence
of (2) and (3).

Next, assume additionally that $F_{M}\in\kOS$ (condition (ii)) holds.
Again, by using the expression (\ref{eq:RandomSumFM}) and Proposition~\ref{lem:Zufaellige-Sum-B-2}(b) we obtain $\overline{F_{M}}\asymp
\overline{G}$
and hence the equivalence of (2) and (3).

Finally, under condition $F_{I}\in\kB\cap\mathcal{DK}$ we can use
Theorem~\ref{lem:YangLemma2}. By choosing $F_{i}=F_{\bar{X}}$ it is easy
to see that we can find appropriate constants $\alpha,\beta$ such
that $\intop_{-\infty}^{\infty}(y\vee-\beta)\,\mathrm{d}F_{\bar
{X}}(y)\leq-\alpha$
holds. Hence, there exists a constant $r$ such that $\overline
{F_{M}}(x)\leq r\overline{F_{I}}(x)$
for all $x\in(-\infty,\infty)$. By (a) of Theorem~\ref{thm:P-V-E-for-OS}
we obtain $\overline{F_{M}}\asymp\overline{F_{I}}$ and hence
$F_{M},G\in
\kB$.
Now, $\overline{F_{M}}\asymp\overline{G}$ follows from $\overline
{F_{I}}\asymp\overline{G}$
and $\overline{F_{M}}\asymp\overline{F_{I}}$.
\end{pf*}


\section*{Acknowledgements}

The authors wish to thank Sergey Foss for interesting discussions
and various suggestions for improvements of the paper, including the
suggestion
to denote the new class by $\mathcal{J}$ (for \emph{jump}). They also
thank Yuebao
Wang for detecting an error in an example in an earlier version of the
manuscript (which
lead to the new paper \cite{Xu2014}).
S. Beck was supported by `Verein zur F\"orderung der
Versicherungswissenschaften an den Berliner Universit\"aten e.V.'.


%

\printhistory
\end{document}